\documentclass[10pt,a4paper]{article}
\usepackage{graphicx, amsmath,amsthm,amssymb,thmtools} 
\usepackage[dvipsnames]{xcolor}

\usepackage{a4wide}

\usepackage{hyperref}
\usepackage{cleveref}
\usepackage{soul}

\renewcommand{\phi}{\varphi}

\newcommand{\floor}[1]{\left\lfloor #1 \right\rfloor}

\newcommand{\conf}{configuration}
\newcommand{\Ss}{\mathcal{S}}

\newcommand{\Cc}{\mathcal{C}}
\newcommand{\Pp}{\mathcal{P}}
\newcommand{\Uu}{\mathcal{U}}
\newcommand{\Hh}{\mathcal{H}}
\newcommand{\Ii}{\mathcal{I}}

\usepackage{xspace}

\usepackage{here}

\newtheorem{theorem}{Theorem}
\newtheorem{lemma}[theorem]{Lemma}
\newtheorem{corollary}[theorem]{Corollary}

\title{The $\chi$-binding function of $d$-directional segment graphs}
\author{
Lech Duraj\thanks{Theoretical Computer Science Department, Faculty of Mathematics and Computer Science,
Jagiellonian University, Krak\'{o}w, Poland. Partially supported by National Science Center of Poland grant 2019/34/E/ST6/00443. \protect}
 \and
Ross J. Kang\thanks{Korteweg--de Vries Institute for Mathematics, University of Amsterdam, Netherlands. Partially supported by a Vidi grant (639.032.614) of the Dutch Research Council (NWO) and the Gravitation programme NETWORKS (024.002.003) of the Dutch Ministry of Education, Culture and Science (OCW). \protect\href{mailto:r.kang@uva.nl}{\protect\nolinkurl{r.kang@uva.nl}}}
 \and
Hoang La\thanks{Theoretical Computer Science Department, Jagiellonian University, Krak\'{o}w, Poland. Partially supported by a Polish National Science Center grant BEETHOVEN; UMO-2018/31/G/ST1/03718. \protect\href{mailto:hoang.la.research@gmail.com}{\protect\nolinkurl{hoang.la.research@gmail.com}} }
 \and 
Jonathan Narboni\footnotemark[1]
 \and 
Filip Pokr\'{y}vka\thanks{Faculty of Informatics, Masaryk University, Brno, Czech Republic. \protect\href{mailto:xpokryvk@fi.muni.cz}{\protect\nolinkurl{xpokryvk@fi.muni.cz}} }
 \and 
Cl\'ement Rambaud\thanks{Université Côte d'Azur, CNRS, Inria, I3S, Sophia-Antipolis, France. Supported by the ANR project DIGRAPHS (ANR-19-CE48-0013). \protect\href{mailto:clement.rambaud@inria.fr}{\protect\nolinkurl{clement.rambaud@inria.fr}}}
 \and 
Amadeus Reinald\thanks{LIRMM, CNRS, Universit\'e de Montpellier, France. Supported by the ANR project DIGRAPHS (ANR-19-CE48-0013). \protect\href{mailto:amadeus.reinald@lirmm.fr}{\protect\nolinkurl{amadeus.reinald@lirmm.fr}}}
}
\date{\today}

\begin{document}

\maketitle

\begin{abstract}
    Given a positive integer $d$, the class $d$-DIR is defined as all those intersection graphs formed from a finite collection of line segments in ${\mathbb R}^2$ having at most $d$ slopes. Since each slope induces an interval graph, it easily follows for every $G$ in $d$-DIR with clique number at most $\omega$ that the chromatic number $\chi(G)$ of $G$ is at most $d\omega$. 
    We show for every even value of $\omega$ how to construct a graph in $d$-DIR that meets this bound exactly. 
    This partially confirms a conjecture of Bhattacharya, Dvo\v{r}\'ak and Noorizadeh.
    Furthermore, we show that the $\chi$-binding function of $d$-DIR is $\omega \mapsto d\omega$ for $\omega$ even and $\omega \mapsto d(\omega-1)+1$ for $\omega$ odd.
    This extends an earlier result by Kostochka and Ne\v{s}et\v{r}il, which treated the special case $d=2$.
\end{abstract}


\section{Introduction}

In structural graph theory, a fundamental task is to characterise the complex, global parameter of chromatic number $\chi$ in terms of the simpler, more local parameter of clique number $\omega$. For a given graph class, the question of whether this task is even well-defined belongs to the theory of $\chi$-boundedness, an area systematically initiated by Gy\'arf\'as in the 1980s~\cite{Gya87}.
A graph class $\mathcal{G}$ is called {\em $\chi$-bounded} if there is some function $f_\mathcal{G}$ such that $\chi(G) \le f_\mathcal{G}(\omega(G))$  for any $G\in\mathcal{G}$, and, if that is so, the optimal choice of $f_\mathcal{G}$ is called the {\em $\chi$-binding} function of $\mathcal{G}$.
Note that the triangle-free graphs of arbitrarily large chromatic number give rise to many interesting graph classes that are not $\chi$-bounded.
This area of mathematics is deep and active, but despite many recent advances, there are many important classes for which the question of $\chi$-boundedness is difficult and remains open; for a nice recent account of the state of the art, see~\cite{ScSe20survey}. Suffice it to say that determination of the $\chi$-binding function for a nontrivial (non-perfect) class is generally considered a rarity.

Since even before Gy\'arf\'as's early work (see~\cite{AsGr60}),
much attention has been focused on intersection classes. These are graph classes defined by taking some natural collection of sets, usually geometrically-defined, and forming for every finite subcollection of those sets an auxiliary graph in which each vertex corresponds to a set, and two vertices are connected by an edge if and only if the corresponding sets have a nontrivial intersection. Such a restriction of focus is not too confining, not only because many intersection classes are fundamental to the structural understanding of graphs (consult, e.g.~\cite{McMc99,Spi03} for more context), but also because intersection classes are more than rich enough for its $\chi$-boundedness theory to contain many fascinating challenges.

As one prominent (and pertinent) example, consider the collection of straight line segments in the plane $\mathbb{R}^2$ (or more formally, the collection of those closed intervals drawn between some pairs of points in $\mathbb{R}^2$). Its intersection class is called the {\em segment (intersection) graphs}. If all the segments happen to lie in parallel, then the resulting segment graph is an interval graph, and thus perfect; that is, it has equality between $\chi$ and $\omega$.
One should not expect this to remain true when we allow the segment slopes to vary, but it is reasonable to ask if there might remain a good relationship between $\chi$ and $\omega$, that is, are segment graphs $\chi$-bounded? Indeed, Erd\H{o}s asked this in the 1970s (see~\cite{Gya87,PKKLMTW14} for more details of the problem's provenance).

Despite sustained attention and nice partial results (see e.g.~\cite{KoNe95,McG00,Suk14}), Erd\H{o}s's innocent-looking but very difficult challenge was not resolved until 2014, and in the negative. Work of Pawlik, Kozik, Krawczyk, Laso\'n, Micek, Trotter, and Walczak~\cite{PKKLMTW14} provided a strikingly elegant construction of triangle-free segment graphs of arbitrarily large chromatic number. Interestingly, the graphs produced by the construction are isomorphic to the intersection graphs that were exhibited by Burling~\cite{Bur65thesis} to show that the intersection class for axis-aligned boxes in ${\mathbb R}^3$ is not $\chi$-bounded.

One might consider this a coda, but permit us to prolong the narrative, in particular by parameterising according to the number of segment slopes.  We first note that the construction of Pawlik {\em et al.}~must have arbitrarily many slopes. Suppose to the contrary that the segments admit at most $d$ slopes. By the same observation as above, each slope induces a perfect graph $H$ having $\chi(H)$ equal to $\omega(H)\le 2$. By colouring the slopes disjointly from one another, we can conclude that $\chi$ is at most $2 d$ for the entire graph, a contradiction.

More generally, given a positive integer $d$, we refer to the {\em $d$-directional segment graphs} as those intersection graphs formed from a collection of line segments in $\mathbb{R}^2$ having at most $d$ distinct slopes. Call the class of such graphs {\em $d$-DIR}. (Take note of the subtlety that this is not {\em per se} an intersection class, but rather a union of many of them.)
What we just argued is that $\chi(G) \le d\omega(G)$ for any $G$ in $d$-DIR. Thus, in a simple fashion, we can conclude that the graph class $d$-DIR is $\chi$-bounded.

The question we would like to address here is, what is the $\chi$-binding function of $d$-DIR? This question was raised
first around the turn of the century by Kostochka and Ne\v{s}et\v{r}il~\cite{Kostochka2002} and again more
recently by Bhattacharya, Dvo\v{r}\'ak and Noorizadeh~\cite{BDN23eurocomb}. 
In both works, it was
proved that there are $2$-directional segment graphs of clique number $2t$ and chromatic number $4t$, for all $t\in \mathbb{N}$, which attains the bound of the above simple argument in this special case.  Bhattacharya {\em et al.}~moreover conjectured that that simple bound is optimal in general, that is, that $\omega \mapsto d \omega$ is the $\chi$-binding function of $d$-DIR for all $d>2$.

We succeeded in completely resolving this question as follows.
\begin{theorem}\label{thm:main}
The $\chi$-binding function of $d$-DIR is
\begin{align*}
\omega\mapsto
\begin{cases}
d\omega & \text{if $\omega$ is even, or}\\
d(\omega-1)+1 & \text{if $\omega$ is odd.}
\end{cases}
\end{align*}
\end{theorem}
\noindent
One can interpret this statement as a graceful generalisation of perfection in the $d=1$ case.
We note that Kostochka and Ne\v{s}et\v{r}il~\cite{Kostochka2002} had already established the $d=2$ case of Theorem~\ref{thm:main}, unbeknownst to the authors in~\cite{BDN23eurocomb}. This had already refuted the conjecture in~\cite{BDN23eurocomb} in this special case.

The crux in Theorem~\ref{thm:main} is to prove the lower bound in the even $\omega$ case. Here is a rough sketch of that construction, which at a high level combines aspects of both the construction of Pawlik {\em et al.}~and that of Bhattacharya {\em et al.}, and indeed strengthens/refines both constructions.
We start by fixing $\omega=2$ and proceed by induction on $d$, exhibiting triangle-free graphs with $t$-fold chromatic number at least $2td$.
We refine the induction used by Pawlik {\em et al.}~by controlling, at every step, the growth of the number of slopes required for achieving a given $t$-fold chromatic number. We then obtain the result for all even $\omega>2$ by blowing up each segment by a factor $\omega/2$.

The structure of the paper is as follows.
In Section~\ref{sec:prelim}, we lay the ground for our proof by introducing the notions and tools used throughout the paper.
Then we settle the even $\omega$ case in Section~\ref{sec:even}.
The lower bound for the odd $\omega$ case  in Section~\ref{sec:oddlower} is a mild adaptation of the even $\omega$ construction.
We give a short proof of the upper bound for the odd $\omega$ case in Section~\ref{sec:oddupper}.
We end with some open questions for future research in Section~\ref{sec:conclusion}.

\section{Preliminaries}\label{sec:prelim}

\subsection{\texorpdfstring{$t$}{t}-fold colourings}

For all positive integers $a$ and $t$, a \emph{$t$-fold $a$-colouring} of a graph $G$ is a function $\phi\colon V(G) \to \binom{\{1,\dots,a\}}{t}$ such that for every edge $uv \in E(G)$, $\phi(u) \cap \phi(v) = \emptyset$.
For $t=1$, we will often refer to $\phi$ simply as an \emph{$a$-colouring}.
For any $t$-fold colouring $\varphi$ of a graph $G$, if $X$ is a set of vertices of $G$, we denote by $\varphi(X)$ the set of colours used to colour the vertices in $X$, i.e.~$\varphi(X) = \bigcup_{v\in X}\varphi(v)$.
We will also use the notation $\varphi_{\vert X}$ to denote the restriction of the function $\varphi$ to the set $X$.

\subsection{Segments}
A \emph{segment} $S$ with endpoints $A=(x_1,y_1)$ and $B=(x_2,y_2)$ in $\mathbb{R}^2$ is the set $\{A+\lambda (B-A) \mid \lambda \in [0,1]\}$. The \emph{length of a segment} is its Euclidean norm. 
A \emph{rectangle} $R$ is defined as the Cartesian product of two segments $[a,c] \times [b,d]$ with $a,b,c,d \in \mathbb{R}$ and $c>a$ and $d>b$. A rectangle always has non-zero area.
If the two intervals have the same length, we say that $R$ is a \emph{square}. We naturally define the \emph{left}, \emph{right}, \emph{top} and \emph{bottom} \emph{sides} of a rectangle $R$ as the four segments whose union is the boundary of $R$. The left and right side of $R$ are the vertical sides of $R$, and the top and bottom sides of $R$ are the horizontal sides of $R$. The length of a vertical side is the \emph{height} of $R$, the length of a horizontal side is the \emph{width} of $R$. 
The aspect ratio of $R$ is the height divided by the width of $R$. A segment $S$ \emph{crosses a rectangle $R$ vertically or horizontally} if $S$ intersects the two horizontal or vertical sides of $R$, respectively. We say that a rectangle $R_1$ crosses a rectangle $R_2$ vertically (respectively, horizontally) if the left and right sides (respectively, the top and bottom sides) of $R_1$ cross the rectangle $R_2$ vertically (respectively, horizontally). If $R_1$ crosses $R_2$ vertically, then $R_2$ crosses $R_1$ horizontally.
The \emph{slope} of a segment in $\mathbb{R}^2$ with endpoints $(x_1,y_1),(x_2,y_2)$ is $\frac{y_2-y_1}{x_2-x_1}$ if $x_1\neq x_2$,
and $\infty$ otherwise.
The \emph{slope number} of a finite family $\mathcal{S}$ of segments is the cardinality of the set of the slopes of segments in $\mathcal{S}$.

Given a set $\mathcal{S}$ of segments, the \emph{intersection graph} $G(\mathcal{S})$ of $\mathcal{S}$ is the graph with vertex set $\mathcal{S}$ and with edges consisting of all pairs $SS'$ with $S, S' \in  \mathcal{S}$ such that $S \cap S' \neq \emptyset$. 

A \emph{$t$-blowup of a segment} 
is a multiset of $t$ copies of the same segment. A \emph{$t$-blowup of a set of segments} is the multiset of the $t$-blowup of the segments. A \emph{$t$-blowup of a graph} is the graph obtained from $G$ by replacing every vertex by a copy of $K_t$, and every edge by a copy of $K_{t,t}$. Observe that a $t$-fold colouring of $G$ corresponds to a colouring of the $t$-blowup of $G$. 

\subsection{Configurations}

We now introduce the notion of probes, first used in~\cite{PKKLMTW14}.
Let $\mathcal{S}$ be a family of segments in \emph{the unit square} $\Uu = [0,1] \times [0,1]$.
Let $P = [a,c] \times [b,d]$ be a rectangle included in $\Uu$.
We denote by $\mathcal{S}(P)$ the set of segments in $\mathcal{S}$ intersecting $P$.
The \emph{right-extension} of $P$ is the rectangle $[a,1] \times [b,d]$.
The rectangle $P = [a,c] \times [b,d]$ is a \emph{probe} for $\mathcal{S}$ if the following hold:
\begin{itemize}
    \item $c=1$, that is, the right side of $P$ lies on the boundary of $\Uu$;
    \item all segments in $\mathcal{S}(P)$ cross $P$ vertically; and
    \item all segments in $\mathcal{S}(P)$ are pairwise disjoint.
\end{itemize}
Given a probe $P$, a \emph{root} of $P$ is a rectangle $[a,c'] \times [b,d]$ where $c'$ is a real number such that $[a,c'] \times [b,d]$ is disjoint from every segment in $\mathcal{S}$.

A \emph{\conf} is a pair $\mathcal{C} = (\mathcal{S},\mathcal{P})$ with $\mathcal{S}$ a set of segments in $\Uu$ and $\mathcal{P}$ a family of pairwise disjoint probes for $\mathcal{S}$.

For convenience, we define the intersection graph $G(\mathcal{C})$ of a configuration $\mathcal{C}=(\mathcal{S},\mathcal{P})$ as $G(\mathcal{S})$. We say that $\Cc$ is triangle-free if $G(\Ss)$ is a triangle-free graph.

Let $R$ be a square inside $\Uu$ and let $\psi$ be the positive homothety mapping $\Uu$ to $R$. The \emph{$R$-scaled copy} of $\Cc$ is the configuration $\Cc' = (\Ss',\Pp')$ where $\mathcal{S}'$ is the set of images of the segments in $\mathcal{S}$ by $\psi$ and
$\mathcal{P}'$ is the set of the right-extensions
of the images by $\psi$ of the probes in $\Pp$.
Observe that $\Ss'$ has the same set of slopes as $\Ss$.

\subsection{Copying configurations}\label{subsec:copying}

Our refinement on the construction of~\cite{PKKLMTW14} is obtained at the cost of adding significantly many more copies of the construction achieving $\chi$ colours when building the graph requiring $\chi+1$ colours. Here, we describe this operation.
For every configuration $\Cc$ and for every positive integer $k$, we define the new configuration $\Cc^{(k)}$ with segments set $\mathcal{S}^{(k)}$ and probes set $\mathcal{P}^{(k)}$ as ``a way of joining $k$ independent copies of $\Cc$ together''.
With every probe $P \in \mathcal{P}^{(k)}$, we associate a set of $k$ sub-rectangles of $P$ that we call the \emph{pillars} of $P$.
The pillars have the same height as $P$, are pairwise disjoint, and no pillar starts from the left side of $P$ so that there is a root of $P$ not belonging to any pillar.
Informally, each pillar will correspond to a copy of $\Cc$.

Formally, $\Cc^{(k)}=(\Ss^{(k)},\Pp^{(k)})$ is defined as follows:
\begin{itemize}
    \item If $k=1$, then $\mathcal{C}^{(1)} = \mathcal{C}$. For every probe $P$, we define its only pillar $I_0^P$ as the minimal rectangle containing $P$ minus a root. Every segment of $\Ss(P)$ is a segment of $\Ss(I_0^P)$ and it also crosses $I_0^P$ vertically.
    \item If $k\geq 2$, then for every probe $P \in \mathcal{P}^{(k-1)}$ we choose a square $R_P$ in a root of $P$ disjoint from its pillars and insert an $R_P$-scaled copy $\mathcal{C}_P = (\mathcal{S}_P,\mathcal{P}_P)$ of $\mathcal{C}$.
    We define $\mathcal{S}^{(k)} = \mathcal{S}^{(k-1)} \cup  \bigcup_{P \in \mathcal{P}^{(k-1)}} \mathcal{S}_P$ and $\mathcal{P}^{(k)} = \bigcup_{P \in \mathcal{P}^{(k-1)}} \mathcal{P}_P$. 
    In other words, instead of $P$, we take new probes, generated in $R_P$ by the new copy of $\Cc$, extended to the right. These new probes are then wholly contained in $P$.
    
    For every such probe $P'$, we define $k$ different pillars of $P'$. There are $k-1$ pillars $I_0^{P'},I_1^{P'},\dots,I_{k-2}^{P'}$ which are intersections between the pairwise disjoint pillars $I_0^{P},I_1^{P},\dots,I_{k-2}^{P}$ of $P$ and $P'$ and one pillar $I_{k-1}^{P'}$ that is the minimal rectangle containing $R_P \cap P'$ minus a root of $P'$. Since $R_P$ is contained in a root of $P$, $I_{k-1}^{P'}$ is also disjoint from $I_0^{P'},\dots,I_{k-2}^{P'}$. Furthermore, for every $i\in \{0,\dots,k-2\}$, since segments in $\Ss^{(k-1)}(I_i^P)$ cross $I_i^P$ vertically, they also cross $I_i^{P'}$ vertically. Therefore, $\Ss^{(k-1)}(I_i^P)\subseteq \Ss^{(k)}(I_i^{P'})$.  
\end{itemize}
Note that $G(\Cc^{(k)})$ is the union of disjoint copies of $G(\Cc)$,
and in particular $\Cc^{(k)}$ is triangle-free if $\Cc$ is triangle-free. See \Cref{fig:Ck}.

\begin{figure}[H]
    \centering
    \includegraphics{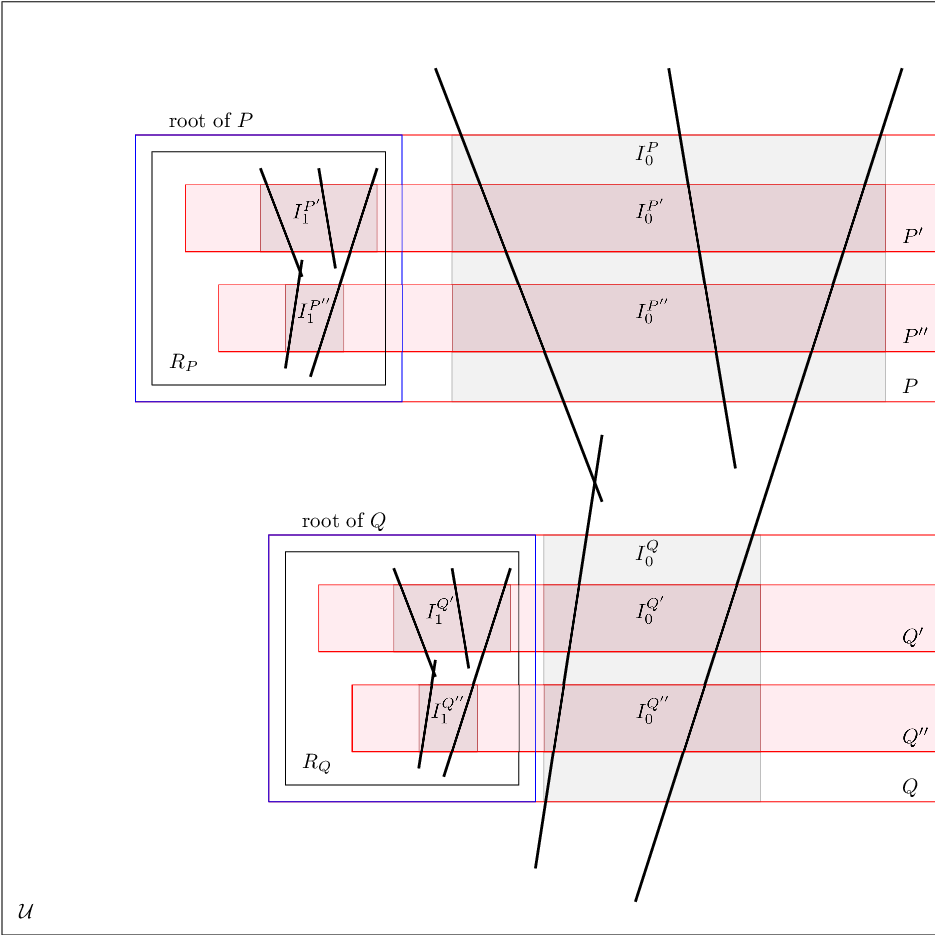}
    \caption{Building $\mathcal{C}^{(2)}$ from $\Cc^{(1)}=\Cc$.}
    \label{fig:Ck}
\end{figure}


\begin{lemma} \label{lemma:pillars}
    Let $\Cc=(\Ss,\Pp)$ be a configuration and let $t$ be a positive integer. If there exists a nonnegative integer $s$ such that, for every $t$-fold colouring $\phi$ of $G(\Cc)$, there exists a probe $P\in \Pp$ such that $|\phi(\Ss(P))|\geq s$, then, for every $t$-fold colouring $\phi$ of $G(\Cc^{(k)})$, there exists a probe $P'$ such that, for every pillar $I$ of $P'$, $|\phi(\Ss(I))|\geq s$.
\end{lemma}

\begin{proof}
    The proof is an induction on $k$. 
    
    Suppose that $k=1$. Observe that $\Cc^{(1)}=\Cc$. Let $\phi$ be a $t$-fold colouring of $G(\Cc^{(1)})$. There exists a probe $P$ such that $\big|\phi(\Ss(P))\big|\geq s$ by assumption on $\Cc$. Since every segment of $\Ss(P)$ crosses the only pillar $I$ of $P$ vertically, $\big|\phi(\Ss(I))\big|\geq s$. Therefore, $P'=P$ satisfies the conclusion of \Cref{lemma:pillars}. 
    
    Suppose that $k\geq 2$. 
    Let $\phi$ be a $t$-fold colouring of $G(\Cc^{(k)})$. By the inductive hypothesis, there exists a probe $P$ of $\Cc^{(k-1)}$ such that for every $i\in \{0,\dots,k-2\}$, every pillar $I^P_i$ of $P$ satisfies $\big|\,\phi_{\vert \Ss^{(k-1)}}(\Ss^{(k-1)}(I_i^P))\,\big|\geq s$.
    By the definition of $\Cc^{(k)}$, there exists a square $R_P$ in a root of $P$ with an $R_P$-scaled copy $\Cc_P=(\Ss_P,\Pp_P)$ of $\Cc$. By assumption on $\Cc$, there exists a probe $P'\in \Pp_P \subseteq \Pp^{(k)}$ such that $\big|\,\phi_{\vert \Ss_P}(\Ss_P(P'))\,\big|\geq s$. 
    Since every segment of $\Ss_P(P')$ crosses $I^{P'}_{k-1}$ vertically, $\big|\,\phi(\Ss^{(k)}(I^{P'}_{k-1}))\,\big|\geq \big|\,\phi_{\vert \Ss_P}(\Ss_P(I^{P'}_{k-1}))\,\big| \geq \big|\,\phi_{\vert \Ss_P}(\Ss_P(P'))\,\big|\geq s$. 
    For every $i\in \{0,\dots,k-2\}$, every segment of $\Ss^{(k-1)}(I_i^P)\subseteq \Ss^{(k)}(I_i^P) $ crosses the pillar $I^{P'}_i$ vertically. Therefore, we also have $\big|\,\phi(\Ss^{(k)}(I^{P'}_i))\,\big| \geq \big|\,\phi_{\vert \Ss^{(k-1)}}(\Ss^{(k-1)}(I_i^P))\,\big|\geq s$.
\end{proof}

\section{Lower bound for even \texorpdfstring{$\omega$}{omega}}\label{sec:even}

In this section, we show the lower bound for the $\chi$-binding function of $d$-DIR for even clique numbers.
First, we exhibit, for any $t$ and $d$, a triangle-free $d$-DIR configuration with $t$-fold chromatic number at least $2td$. Then, a $t$-blowup of this construction yields one with clique number $2t$ and chromatic number at least $2td$, establishing the even case in Theorem~\ref{thm:main}.

\begin{theorem}\label{thm:main_construction_omega_even}
    For all positive integers $t$ and $d$, there exists a triangle-free configuration $\mathcal{C}_{t,d} = (\mathcal{S}_{t,d}, \mathcal{P}_{t,d})$
    with slope number at most $d$ such that for every $t$-fold colouring $\phi$ of $G(\mathcal{C}_{t,d})$, there exists a probe $P \in \mathcal{P}_{t,d}$ such that $|\phi(\mathcal{S}_{t,d}(P))| \geq 2td$.
\end{theorem}

\begin{proof}
    We proceed by induction on $d$.
    For $d=0$, the configuration $(\emptyset, \{\mathcal{U}\})$
    satisfies the statement of the theorem.
    
    Suppose that $d>0$ and that $\mathcal{C}_{t,d-1}$ is already constructed.
    Let $\mathcal{H} = \mathcal{C}_{t,d-1}^{(4t+1)}=(\mathcal{S}_\mathcal{H},\mathcal{P}_\mathcal{H})$
    be the copy of $4t+1$ of these configurations as defined in Subsection~\ref{subsec:copying}.
    Let $\Gamma$ be the set of the slopes of the segments in $\mathcal{S}_{t,d-1}$. We define the slope $\gamma$ as the minimum aspect ratio over all probes of $\Pp_\Hh$, divided by $4t$. 

    We will construct a family of 
    triangle-free
    configurations $\mathcal{H}_i = (\mathcal{S}_{i},\mathcal{P}_{i})$ for all $i \geq 0$ such that
    \begin{enumerate}
        \item for every segment $S \in \mathcal{S}_{i}$, the slope of $S$ is in $\Gamma \cup \{\gamma\}$, and
        \label{item:induction_i_slopes}
        \item for every $t$-fold colouring $\phi$ of $G(\mathcal{H}_{i})$, there is a probe $P \in \mathcal{P}_i$
        such that $|\phi(\mathcal{S}_i(P))| \geq 2t(d-1)+2t(1-\frac{1}{2^{i}})$. \label{item:induction_i_number_of_colours}
    \end{enumerate}
    For $i = \floor{\log (2t)}+1$, \Cref{item:induction_i_number_of_colours} implies that $|\phi(\mathcal{S}_i(P))| \geq 2td$, and so it suffices to take $\mathcal{C}_{t,d} = \mathcal{H}_i$.

    We build $\Hh_i$ as follows.
    For $i=0$, we take $\mathcal{H}_0 = \mathcal{C}_{t,d-1}$. \Cref{item:induction_i_slopes} holds by definition of $\Hh_0$ and \Cref{item:induction_i_number_of_colours} holds by the inductive hypothesis on $\Cc_{t,d-1}$. 
    
    Now, assume that $i>0$. We will build $\Ss_i$ and $\Pp_i$ iteratively.
    Initialise $\mathcal{S}_i$ to $\mathcal{S}_{i-1}$ and $\mathcal{P}_i$ to $\emptyset$.
    For each probe $P\in \Pp_{i-1}$, let $R$ be a square in a root of $P$, and add an $R$-scaled copy $\Hh_P$ of $\Hh$ in $R$. 
    At this point, the configuration is still triangle-free as it is a disjoint union of $\Hh_{i-1}$ and copies of $\Hh$.
    For every probe $Q$ in $\Hh_P$, we divide $Q$ vertically into $4t$ rectangles of equal height with the same width as $Q$. We call these rectangles the \emph{layers} of $Q$. Let $(L_j(Q))_{j\in \{0,\dots,4t-1\}}$ be the layers of $Q$ starting from the top. 

    First, we add the top layer $L_0(Q)$ to $\mathcal{P}_i$ as a single probe. Then, for every $j\in \{1,\dots,4t-1\}$, we add two segments in the interior of $L_j(Q)$ to $\Ss_i$, with slope $\gamma$, that intersect at exactly one point: $D^j_1$ crossing exactly the pillars $I^Q_{j+1},\dots,I^Q_{4t}$ horizontally, and $D^j_2$ crossing $I^Q_j$ horizontally and no other pillars. Since the slope of the diagonal of a layer is at least $\gamma$, such segments are well-defined. 
    Moreover, the configuration is still triangle-free. On the one hand, any $D^j_1$ or $D^j_2$ intersects a set of disjoint segments over all pillars. On the other, any pillar, and thus any segment crossing it vertically, cannot intersect both $D^j_1$ and $D^j_2$.
    See \Cref{fig:layers}.
    
    \begin{figure}[H]
        \centering
        \includegraphics{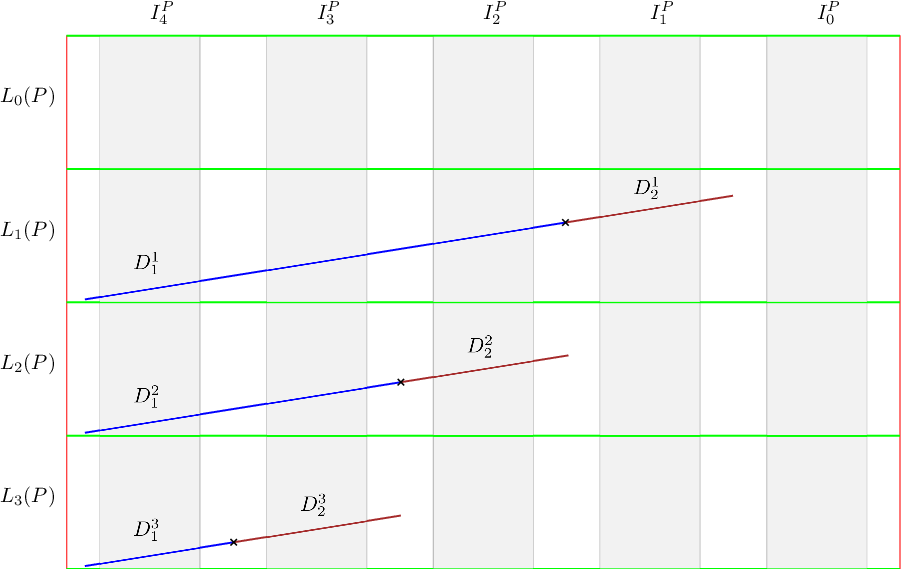}
        \caption{Layers of a probe $P$ drawn in green, with new segments $D_1^j$ and $D_2^j$. Example for $t=1$.}
        \label{fig:layers}
    \end{figure}
    
    Finally, we add to $\mathcal{P}_i$ two probes $P_{D^j_1},P_{D^j_2}$, with $P_{D^j_1}$ crossed vertically by exactly $D^j_1$ and $I^Q_0,\dots,I^Q_j$, and $P_{D^j_2}$ crossed vertically by exactly $D^j_2$ and $I^Q_0,\dots,I^Q_{j-1}$. See \Cref{fig:new_probes}.
    
    \begin{figure}[H]
        \centering
        \includegraphics{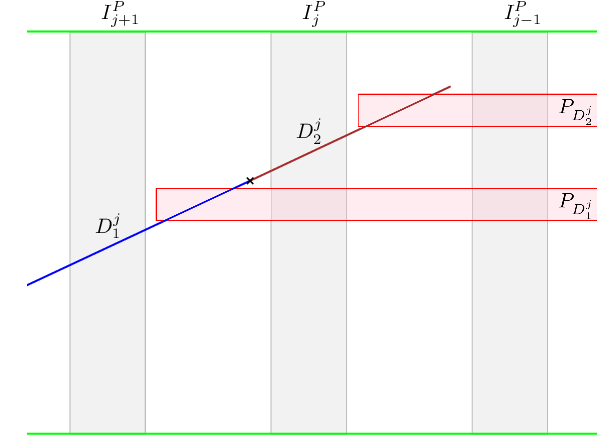}
        \caption{The two new probes $P_{D^j_1}$ and $P_{D^j_2}$ inside the layer $L_j(P)$.}
        \label{fig:new_probes}
    \end{figure}

    It follows from the construction that $\mathcal{P}_i$ is a set of probes for $\mathcal{S}_i$, and that every segment in $\mathcal{S}_i$ has its slope in $\Gamma \cup \{\gamma\}$ so \Cref{item:induction_i_slopes} holds.
    It remains to prove that for every $t$-fold colouring $\phi$, there is a probe $P \in \mathcal{P}_i$ such that $|\phi(\mathcal{S}_i(P))| \geq 2t(d-1)+2t(1-\frac{1}{2^{i}})$.
    Let $\phi$ be a $t$-fold colouring of $G(\Hh_i)$. By the inductive hypothesis, there exists $P \in \mathcal{P}_{i-1}$
    such that $|\phi(\mathcal{S}_{i-1}(P))| \geq 2t(d-1)+2t(1-\frac{1}{2^{i-1}})$.
    By \Cref{lemma:pillars} applied to $\Hh = \mathcal{C}_{t,d-1}^{(4t+1)}$, there exists a probe $Q$ of $\mathcal{H}_P$ such that every pillar $I^Q_j$ satisfies $|\phi(\Ss_i(I^Q_j))|\geq 2t(d-1)$. For every $j \in \{0,\dots,4t-1\}$, let $\Phi_j = \phi(\Ss_i(I^Q_j))$.
    Recall that the top layer $L_0(Q)$ was added as a probe, and it crosses all the $I^Q_j$s horizontally. If it contains at least $2t(d-1)+2t(1-\frac{1}{2^{i}})$ colours, then this probe satisfies \Cref{item:induction_i_number_of_colours}. So, from now on, we may assume that $Q$ contains fewer than $2t(d-1)+2t(1-\frac{1}{2^{i}})$ colours. 
    Since $|\Phi_0| \geq 2t(d-1)$, we have $2t(d-1)+2t(1-\frac{1}{2^{i}}) - |\Phi_0| < 2t$. 
    Therefore, at most $2t-1$ indices $j \in \{1,\dots,4t-1\}$ are such that the set $\Phi_j$ contains a colour not in $\bigcup_{j'\leq j-1} \Phi_{j'}$.
    Thus, for at least $4t-1-2t+1=2t$ indices $j \in \{1,\dots,4t-1\}$, $\Phi_{j} \subseteq \bigcup_{j'\leq j-1} \Phi_{j'}$.
    Symmetrically, for at least $2t$ indices $j \in \{1,\dots,4t-1\}$, we have $\Phi_{j} \subseteq \bigcup_{j'\geq j+1} \Phi_{j'}$.
    Thus, there is an index $j_0 \in \{1,\dots,4t-1\}$ such that
    $\Phi_{j_0} \subseteq \bigcup_{j'\geq j_0+1} \Phi_{j'}$ and $\Phi_{j_0} \subseteq \bigcup_{j' \leq j_0-1} \Phi_{j'}$.
    Let $A = \Phi_{j_0}$, $B = \bigcup_{j' \leq j_0-1} \Phi_{j'}$
    and $C = \bigcup_{j' \geq j_0+1} \Phi_{j'}$. We have $A \subseteq B$ and $A \subseteq C$.
    Consider the layer $L_{j_0}(Q)$ and the corresponding segments $D_1^{j_0},D_2^{j_0}$. Since $D_1^{j_0}$ and $D_2^{j_0}$ share a common point, $\phi(D_1^{j_0}) \cap \phi(D_2^{j_0}) = \emptyset$. Since $D_1^{j_0}$ crosses $I^Q_{j_0+1},\dots,I^Q_{4t}$ horizontally, we have $\phi(D_1^{j_0}) \cap C = \emptyset$. Since $A \subseteq C$, we also have $\phi(D^{j_0}_1) \cap A = \emptyset$. Moreover, $D_2^{j_0}$ crosses $I^Q_{j_0}$ horizontally, so $\phi(D^{j_0}_2) \cap A = \emptyset$. As a result, $\phi(D^{j_0}_1)$, $\phi(D^{j_0}_2)$, and $A$ are pairwise disjoint.
    Thus, $|\phi(D^{j_0}_1) \cup \phi(D^{j_0}_2) \cup A| = |\phi(D^{j_0}_1)| + |\phi(D_2^{j_0})| + |A| \geq t + t + 2t(d-1) = 2td$.
    Hence, letting $Z_{i-1} = \phi(\mathcal{S}_{i-1}(P))$, we have $|(\phi(D^{j_0}_1) \cup \phi(D^{j_0}_2) \cup A) \setminus Z_{i-1}| \geq 2td - |Z_{i-1}|$. Therefore, there is some $\alpha \in\{1,2\}$ such that $|(\phi(D^{j_0}_\alpha) \cup A) \setminus Z_{i-1}| \geq \frac{2td-|Z_{i-1}|}{2}$. It follows that
    \[
    |\phi(D^{j_0}_\alpha) \cup A \cup Z_{i-1}| \geq \frac{2td+|Z_{i-1}|}{2} \geq 2t(d-1) + 2t\left(1-\frac{1}{2^i}\right).
    \]
    Then $\phi\left(\Ss_i\left(P_{D^{j_0}_\alpha}\right)\right)$ contains at least 
    $|\phi(D^{j_0}_\alpha) \cup B \cup Z_{i-1}| \geq |\phi(D^{j_0}_\alpha) \cup A \cup Z_{i-1}| \geq 2t(d-1) + 2t(1-\frac{1}{2^i})$ colours. 
    Therefore, \Cref{item:induction_i_number_of_colours} always holds for $\mathcal{H}_i$. This concludes the proof of the theorem.
\end{proof}

\begin{corollary}
    For all positive integers $t$ and $d$, there exists a multiset of segments $\Ss'_{t,d} \subseteq \Uu$ with slope number at most $d$ such that
     $\omega(G(\Ss'_{t,d})) = 2t$ and $\chi(G(\Ss'_{t,d}))\geq 2td$.
\end{corollary}

\begin{proof}
    Let $\Cc_{t,d} = (\Ss_{t,d},\Pp_{t,d})$ be a configuration given by \Cref{thm:main_construction_omega_even}. 
    Let $\Ss'_{t,d}$ be a $t$-blowup of $\Ss_{t,d}$,
    that is a configuration obtained from $\Ss_{t,d}$ by replacing every segment $v \in \Ss_{t,d}$ by $t$ identical copies $v_1, \dots, v_t$.
    Since the graph $G(\Ss_{t,d})$ is triangle-free, we have $\omega(G(\Ss'_{t,d})) = 2t$. 
    Let $\phi$ be a colouring of $G(\Ss'_{t,d})$. 
    The colouring $\phi$ corresponds to a $t$-fold colouring $\phi_t$ of $G(\Ss_{t,d})$ defined by $\phi_t(v) = \{\phi(v_1), \dots, \phi(v_t)\}$ for all $v\in V(G(\Ss_{t,d}))$. 
    By \Cref{thm:main_construction_omega_even}, there exists a probe $P\in \Pp_{t,d}$ such that $|\phi_t(\Ss_{t,d}(P))|\geq 2td$. 
    Hence, by definition of $\phi_t$, $\phi$ uses at least $2td$ colours as claimed.
\end{proof}

\section{Lower bound for odd \texorpdfstring{$\omega$}{omega}}\label{sec:oddlower}

In this section, we certify the lower bound for the $\chi$-binding function in the case of odd clique numbers. We rely on the construction of $d$-DIR graphs with $(t+1)$-fold chromatic number at least $2(t+1)d$ as guaranteed by \Cref{thm:main_construction_omega_even}.
Instead of blowing up each segment by $t+1$ as in the even case, many blowups will be $t$-blowups, ensuring the clique number increases only to $2t+1$. We  do so in a way that also ensures the chromatic number is $2td+1$, as desired.

\begin{theorem}\label{thm:lower_bound_odd_omega}
    For every nonnegative integer $t$ and every positive integer $d$, there exists a multiset of segments $\mathcal{S}''_{t,d}$ with slope number at most $d$ such that 
    $\omega(G(\mathcal{S}''_{t,d})) = 2t+1$ and $\chi(G(\mathcal{S}''_{t,d})) \geq 2td+1$.
\end{theorem}

\begin{proof}
    Let $t$ be a nonnegative integer and let $d$ be a positive integer.
    By \Cref{thm:main_construction_omega_even}, there is a triangle-free configuration $\mathcal{C}_{t+1,d} = (\mathcal{S}_{t+1,d}, \mathcal{P}_{t+1,d})$
    with slope number at most $d$
    such that every $(t+1)$-fold colouring of $G(\mathcal{C}_{t+1,d})$ uses at least $2(t+1)d$ colours.
    Let $\mathcal{S}^1, \dots, \mathcal{S}^d$ be the partition of $\mathcal{S}_{t+1,d}$ according to the slopes of the segments.
    For every $i \in [d]$, $G(\mathcal{S}^i)$ is a triangle-free interval graph, and so it is bipartite.
    Hence there is a partition $\mathcal{S}^{i,1}, \mathcal{S}^{i,2}$ of $\mathcal{S}^i$ such that the segments in $\mathcal{S}^{a,i}$ are pairwise disjoint for each $a \in \{1,2\}$.
    Now, let $\mathcal{S}''_{t,d}$
    be obtained from $\mathcal{S}_{t+1,d}$ by blowing up every segment $v$ in $\bigcup_{i \in [d-1]} \mathcal{S}^i$ and $\mathcal{S}^{d,1}$ into $t$ copies $v_1, \dots, v_{t}$, and every segment $v$ in $\mathcal{S}^{d,2}$ into $t+1$ copies $v_1, \dots, v_{t+1}$.

    First, by construction, $\omega(G(\mathcal{S}''_{t,d})) \leq 2t+1$.
    We now show that $\chi(G(\mathcal{S}''_{t,d})) \geq 2td+1$.
    Let $\phi$ be a $k$-colouring of $G(\mathcal{S}''_{t,d})$. 
    Let us introduce new colours $(1,i),(2,i)$, for $i \in [d]$, and suppose they are not in the image of $\phi$.
    For every $i \in [d-1]$ and  every $v \in \mathcal{S}^{i,a}$, where $a \in \{1,2\}$, we
    let $\psi(v) = \{\phi(v_1), \dots,\phi(v_t), (a,i)\}$.
    For every $v \in \mathcal{S}^{d,1}$, we let $\psi(v) = \{\phi(v_1), \dots, \phi(v_t), (1,d)\}$.
    Finally, for every $v \in \mathcal{S}^{d,2}$, we let $\psi(v) = \{\phi(v_1), \dots, \phi(v_{t+1})\}$.
    Note that $\psi(v)$ has size exactly $t+1$ for every $v \in \mathcal{S}_{t+1,d}$, and
    $\psi$ uses at most $k+2d-1$ colours (since we did not use the colour $(2,d)$).
    Moreover, as $\phi$ is proper and we invoke the bipartitions of the $\mathcal{S}^i$s, we deduce that $\psi$ is a $(t+1)$-fold colouring of $\mathcal{S}_{t+1,d}$.
    This implies that $k +2d-1 \geq 2(t+1)d$, and so $\phi$ uses $k \geq 2td+1$ colours.
    This proves that $\chi(G(\mathcal{S}''_{t,d})) \geq 2td+1$.
\end{proof}

\section{Upper bound for odd \texorpdfstring{$\omega$}{omega}}\label{sec:oddupper}

In this section, we upper bound the $\chi$-binding function for $d$-DIR in the case of odd clique numbers, which combined with the result of the previous section settles the odd case in Theorem~\ref{thm:main}.

Let $G$ be an interval graph and
let $\Ii = (I_v \mid v \in V(G))$ be an interval representation of $G$. 
Let $\omega$ be an integer with $\omega \geq \omega(G)$.
For any $x \in \mathbb{R}$, we define the set $S_x=\{u \in V(G) \mid x \in I_u\}$ of vertices corresponding to intervals containing $x$.
A vertex $v$ of $G$ is said to be \emph{$\omega$-uncovered} if there exists $x \in I_v$ such that $|S_x| \leq \frac{\omega - 1}{2}$. 
Observe that $S_x$ consists of a clique of $\omega$-uncovered vertices only. 
We use the following lemma appearing in~\cite{Kostochka2002} as a tool to build colourings of $d$-DIR graphs in the case of odd $\omega$.

\begin{lemma}[{\cite[Lemma~1]{Kostochka2002}}]\label{lemma:colouring_special_vertices_in_interval_graphs}
    Let $G$ be an interval graph with clique number at most $\omega$. There is a colouring $\phi\colon V(G) \to \{0, \dots,\omega-1\}$ of $G$ such that $\phi(v)\neq 0$ for every $\omega$-uncovered vertex $v$.
\end{lemma}

\begin{theorem}
    For any graph $G$ in $d$-DIR, if $\omega(G)$ is odd, then $\chi(G) \le d(\omega(G)-1)+1$.
\end{theorem}

\begin{proof}
    Let $(I_u \mid u \in V(G))$ be a $d$-DIR representation of $G$.
    Let $V_1, \dots, V_d$ be the partition of $V(G)$ according to the slopes of the segments in its $d$-DIR representation.
    For every $i\in \{1,2,\dots,d\}$, $G[V_i]$ is an interval graph and $\omega(G)\geq \omega(G[V_i])$. So by \Cref{lemma:colouring_special_vertices_in_interval_graphs}, 
    there is a colouring $\phi_i$ of $G[V_i]$ using the colours $0,1,2,\dots,\omega(G)-1$ and no $\omega(G)$-uncovered vertices of $G[V_i]$ are coloured $0$. 
    For every $i\in\{1,\dots,d\}$ and for every $v\in V_i$, let $\phi(v)=(i,\phi_i(v))$ if $\phi_i(v) \neq 0$, and $\phi(v)=0$ otherwise. 
    We claim that $\phi$ is a $(d(\omega(G)-1)+1)$-colouring of $G$.
    Indeed, if $uv$ is an edge in $G$ such that $\phi(u)=\phi(v)$, then $\phi(u)=\phi(v)=0$. 
    In other words, $u$ and $v$ are both not $\omega(G)$-uncovered vertices and are in distinct $V_i$s.
    Let $i_u,i_v \in \{1,\dots,d\}$ be such that $u \in V_{i_u}$ and $v \in V_{i_v}$.
    Since $\omega(G)$ is odd, the intersection point $x$ of $I_u$ and $I_v$ is included in at least $\frac{\omega(G)+1}{2}$ intervals corresponding to vertices in $V_{i_u}$ and at least  $\frac{\omega(G)+1}{2}$ intervals corresponding to vertices in $V_{i_v}$.
    This gives a clique in $G$ of order $\omega(G)+1$, a contradiction.
\end{proof}

\section{Concluding remarks}\label{sec:conclusion}


Recall the following definitions.
The {\em Hall ratio} $\rho(G)$ of a graph $G$ is defined by $\rho(G)=\max_{H\subseteq G} \frac{|V(H)|}{\alpha(H)}$, where $\alpha$ denotes the independence number.
The {\em fractional chromatic number} $\chi_f(G)$ of $G$ is the infimum of $a/t$ as $t\to \infty$ with $a$ being the least integer such that $G$ admits a $t$-fold $a$-colouring.
Note that $\omega(G) \le \rho(G)\le \chi_f(G) \le \chi(G)$ for all $G$. 
Walczak~\cite{Wal15} showed that there are triangle-free segment graphs of arbitrarily large $\rho$.
Can we find constructions to certify the same lower bounds as in Theorem~\ref{thm:main} but for $\chi_f$ or $\rho$ instead of $\chi$? In particular, for every $d$ and even $\omega$ is there a $d$-directional segment graph of clique number $\omega$ with $\chi_f \ge d\omega$ or with $\rho \ge d\omega$?

A possible generalisation of our problem is to consider rectangles instead of segments. In a similar manner, we can introduce the notion of the $d$-directional rectangle graphs: consider a set $d$ non-parallel lines and any set of rectangles such that each of them has one of the sides parallel to one of the given lines. We want to find the $\chi$-binding function of such graphs. 

This problem has been quite extensively studied in the case of axis-parallel rectangles (i.e. $d = 1$) and thanks to Chalermsook and Walczak~\cite{ChWa21} it is known that in this case $O(\omega \log \omega)$ colours always suffice; however, the lower bound remains $\Omega(\omega)$. This implies that for any $d$, no more than $O(d \omega \log \omega)$ colours are needed, and our construction from this paper can be easily adapted (by replacing the segments with ``thin enough'' rectangles) to show a lower bound of $\omega \cdot (d + o(1))$. The gap between linear and log-linear function remains open, though, in both $d=1$ and the general case.

\paragraph{Acknowledgement.} This work was initiated during the {\em Sparse (Graphs) Coalition} session, {\em $\chi$-boundedness} in March 2023, organised by James Davies, Bartosz Walczak, and the second author. Our team thanks the organisers and participants for the conducive working atmosphere. We in particular thank Zden\v{e}k Dvo\v{r}\'ak for proposing the problem.
We are grateful to the reviewers for their careful reading and  valuable comments, and especially for pointing us to the  work of Kostochka and Ne\v{s}et\v{r}il~\cite{Kostochka2002} as well as the suggestion of a simpler proof of \Cref{thm:lower_bound_odd_omega}.

\paragraph{Open access statement.} For the purpose of open access,
a CC BY public copyright license is applied
to any Author Accepted Manuscript (AAM)
arising from this submission.

\bibliographystyle{alpha}
\bibliography{biblio}

\end{document}